\documentclass[11pt]{amsart}
\usepackage[margin=1in]{geometry}
\usepackage{scalerel,stackengine, mathrsfs, multicol, enumitem, manfnt, graphicx, setspace, answers, amsmath, amssymb, amscd, amsthm, color, epsfig, url, tikz-cd, mathtools, bm, adjustbox, hyperref, cleveref, array, longtable, multicol, float}
\usetikzlibrary{shapes.geometric}
\hypersetup{colorlinks=true,
    linkcolor=blue,
    filecolor=blue,
    citecolor=blue}

\newtheorem{theorem}{Theorem}[section]
\newtheorem{proposition}[theorem]{Proposition}
\newtheorem{lemma}[theorem]{Lemma}
\newtheorem{corollary}[theorem]{Corollary}

\theoremstyle{definition}

\newtheorem{definition}[theorem]{Definition}

\newtheorem{remark}[theorem]{Remark}

\newtheorem{example}[theorem]{Example}

\renewcommand{\P}{\mathbb{P}}

\newcommand{\calMbar}{\overline{\mathcal{M}}}
\newcommand{\Mbar}{\overline{\mathcal{M}}}
\newcommand{\M}{\mathcal{M}}

\newcommand{\calH}{\mathcal{H}}
\renewcommand{\H}{\mathcal{H}}
\newcommand{\calHbar}{\overline{\mathcal{H}}}
\newcommand{\Hbar}{\overline{\mathcal{H}}}

\newcommand{\Aut}{\operatorname{Aut}}
\newcommand{\Adm}{\operatorname{Adm}}
\newcommand{\Admbar}{\overline{\operatorname{Adm}}}

\newcommand{\del}{{\partial}}

\stackMath
\newcommand\reallywidehat[1]{%
\savestack{\tmpbox}{\stretchto{%
  \scaleto{%
    \scalerel*[\widthof{\ensuremath{#1}}]{\kern.1pt\mathchar"0362\kern.1pt}%
    {\rule{0ex}{\textheight}}
  }{\textheight}%
}{2.4ex}}%
\stackon[-6.9pt]{#1}{\tmpbox}%
}
\tikzset{
  trim node/.default=1cm,
  trim node/.style={
    overlay,
    append after command={
      ([xshift={+#1}]\tikzlastnode.north west)
      ([xshift={+-#1}]\tikzlastnode.south east)}},
  down and trim/.default=1cm,
  down and trim/.style={
    yshift=-(\pgfmatrixcurrentcolumn-1)*1.5\baselineskip,
    trim node={#1}},
  downup and trim/.default=1cm,
  downup and trim/.style={
    yshift=iseven(\pgfmatrixcurrentcolumn) ? -1.5\baselineskip : 0pt,
    trim node={#1}},
  -|/.style={to path={-|(\tikztotarget)\tikztonodes}},
  |-/.style={to path={|-(\tikztotarget)\tikztonodes}},
  -| sl/.style={-|, xslant=-1},
  |- sl/.style={|-, xslant= 1},
  center picture/.style={
    trim left=(current bounding box.center),
    trim right=(current bounding box.center)}}

\begin{document}

\title[Vanishing $H^1$ for $\Hbar_{\underline{3},\underline{g}}(\underline{\mu})$]{
 Vanishing $H^1$ for Hurwitz spaces of fully-marked admissible covers of degree 3
  }
  \author[A.Q.~Li]{Amy Q. Li}
  \address{Department of Mathematics, University of Texas at Austin, United States of America}
  \email{amyqli@utexas.edu}


\begin{abstract}
  We show that the first cohomology group of the Hurwitz space of fully-marked admissible covers $H^1(\Hbar_{\underline{d},\underline{g}}(\underline{\mu}))$ vanishes for covers of degree $ d = 3$ and deduce the same result for the classical Hurwitz space of simply-branched covers.
  In degree 4, we compute examples where $H^1(\Hbar_{\underline{4},\underline{g}}(\underline{\mu}))$ is nonzero, which implies that $H^1(\Hbar_{\underline{d},\underline{g}}(\underline{\mu}))$ is nonvanishing for $d \geq 4$.
  We describe the stratification of the boundary of $\Hbar_{\underline{d},\underline{g}}(\underline{\mu})$ by lower-dimensional $\Hbar_{\underline{d'},\underline{g'}}(\underline{\mu'})$, and set up an inductive framework which may be used for future arguments involving the odd cohomology of $\Hbar_{\underline{d},\underline{g}}(\underline{\mu})$.
\end{abstract}
\maketitle

\section{Introduction}
The Hurwitz space $\Adm^s_{0,b,d}$ parametrizes simply-branched degree-$d$ covers from smooth connected genus-$g$ curves to smooth genus-$0$ curves, marked at $b = 2g + 2d - 2$ branch points.
Harris and Mumford introduced the ``admissible covers'' compactification of $\Adm^s_{0,b,d}$, in which the target and source curves of a cover degenerate to nodal curves as branch points come together and the target curve is stable \cite{HM}.
Although this space is singular, its normalization is the smooth moduli space of twisted stable maps \cite{ACV}, which we denote by $\Admbar^s_{0,b,d}$.
We refer to the normalization whenever we mention the admissible covers compactification.
The Picard Rank Conjecture \cite{DE96} predicts that the rational Picard group vanishes in all degrees $d$ for both $\Adm_{0,b,d}^s$ and its variant $\H^s_{d,g} = \Adm_{0,b,d}^s/S_b$.
Recent work on the geometry of $\H^s_{d,g}$ confirms this conjecture for $d \leq 5$ \cite{DPPic} and for $d > g-1$ \cite{MullanePRC}.
Furthermore, the homology of $\H^s_{d,g}$ stabilizes, implying that the Picard Rank Conjecture holds asymptotically when $b \gg d$ \cite{LLhomstab}.

The Hurwitz space $\Hbar_{2,g}^s$ of degree-2 admissible covers is isomorphic to the quotient $\Mbar_{0,b}/S_b$.
It follows that $H^1(\Hbar_{2,g}^s) = H^1(\Mbar_{0,b})^{S_b} = 0$.
For admissible covers of degree 3, we develop tools for an inductive argument inspired by \cite{ACcoh} to show that $H^1$ vanishes for the Hurwitz space of fully-marked admissible covers $ \Hbar_{\underline{d}, \underline{g}}(\underline{\mu})$, a generalization of $\Admbar^s_{0,b,d}$.
This Hurwitz space parametrizes possibly disconnected admissible covers with fully-marked fibers of prescribed (not necessarily simple) ramification  (see ~\ref{fullmkadm}).

\begin{theorem}
  \label{main}
  The first cohomology group $H^1(\Hbar_{\underline{3}, \underline{g}}(\underline{\mu})) = 0$ for every partition of $3$, every partition of $g$, and any set of $m$ ramification profiles $\mu_i$ partitioning $3$.
\end{theorem}

\noindent We give examples of 1-dimensional Hurwitz spaces in degree 4 and degree 5 of fully-marked admissible covers with positive genus (see \Cref{degree4}).
It follows that $H^1(\Hbar_{\underline{d}, \underline{g}}(\underline{\mu}))$ does not vanish for $d \geq 4$.

The spaces $\Admbar^s_{0,b,d}$ and $\Hbar_{d,g}^s$ are finite group quotients of $\Hbar_{d,g}((2,1,\dots,1)^b)$, where the source curve is connected and all marked fibers have simple ramification (see \Cref{relnotherHurw}).
We immediately obtain the following corollary.

\begin{corollary}
  The first cohomology groups $H^1(\Admbar^s_{0,b,3}) = H^1(\Hbar^s_{3,g}) = 0$.
\end{corollary}

Our proof of \Cref{main} is inspired by the inductive argument introduced by Arbarello and Cornalba to study the low degree odd cohomology of moduli spaces of stable curves $\Mbar_{g,n}$ \cite{ACcoh}.
Arbarello and Cornalba used the stratification of the boundary of $\Mbar_{g,n}$ by the images of gluing maps from products of lower-dimensional moduli spaces of curves to show that $H^k(\Mbar_{g,n})$ vanishes for $k = 1, 3, $ and $5$.
They conjectured that inductive arguments could be used to show vanishing for higher odd $k$, once the requisite base cases were proven.
Indeed, Bergstr\"{o}m, Faber, and Payne showed that $H^k (\Mbar_{g,n})$ vanishes for $k = 7$ and $9$ \cite{BFPOdd}, thereby proving the conjecture for all odd $k < 11$.
This bound is optimal, since the cohomology group $H^{11}(\Mbar_{1,n})$ is nonzero once $n \geq 11$.

Adapting these strategies to Hurwitz spaces requires new tools.
First, we must find a Hurwitz space with the correct boundary structure.
Like $\Mbar_{g,n}$, the boundaries of the simply-branched Hurwitz spaces $\Admbar^s_{0,b,d}$ and $\Hbar_{d,g}^s$ are stratified by the images of gluing maps from products of lower-dimensional Hurwitz spaces.
The Hurwitz spaces appearing in the domains of these gluing maps, however, parametrize admissible covers with arbitrary (not necessarily simple) ramification, since a simply-branched admissible cover in the boundary may have arbitrary ramification over the nodes.
Thus, even if we are primarily interested in the low degree odd cohomology of $\Admbar^s_{0,b,d}$ and $\Hbar_{d,g}^s$, in order to argue inductively we are inevitably led to study the low degree odd cohomology of the more general Hurwitz spaces $\Hbar_{\underline{d},\underline{g}}(\underline{\mu})$.

In \Cref{indodd}, we show that when $k$ is odd, the vanishing of the $k$-th cohomology group of low-dimensional $\Hbar_{\underline{d'}, \underline{g'}}(\underline{\mu'})$ implies the vanishing of the $k$-th cohomology group for all $\Hbar_{\underline{d}, \underline{g}}(\underline{\mu})$.
Thus, for $k = 1$ and $d = 3$, we must compute $H^1(\Hbar_{\underline{d}, \underline{g}}(\mu_1, \mu_2, \mu_3, \mu_4))$ for every possible choice of ramification profile $\mu_i$.
To solve this technical challenge, we take advantage of the natural target map from these 1-dimensional Hurwitz spaces to $\Mbar_{0,4}$, which is a covering map.
For each base case, it suffices to compute the monodromy of this target map, thereby computing that each Hurwitz space is genus-0.
\Cref{main} follows.

In principle, \Cref{indodd} could be used to prove vanishing for higher odd $k$, analogously to the argument in \cite{ACcoh}; see \Cref{futurework}.

\subsection{Outline of the paper}
In Section 2, we define Hurwitz spaces of fully-marked admissible covers and describe their relationships to other Hurwitz spaces in the literature.
In Section 3, we recast the geometric data of an admissible cover as the combinatorial data of monodromy representations, allowing us to combinatorially describe the target map $\Hbar_{\underline{d}, \underline{g}}(\underline{\mu}) \to \Mbar_{0,m}$.
In Section 4, we show that the boundary divisors of $\Hbar_{\underline{d}, \underline{g}}(\underline{\mu})$ are finite quotients of products of lower-dimensional Hurwitz spaces, setting the stage for the inductive proof of vanishing $H^1$ in Section 5.
Finally, in Section 6, we record the data of the base cases necessary to prove that $H^1$ vanishes for admissible covers in degrees 3.
We also provide examples of 1-dimensional Hurwitz spaces of fully-marked admissible covers in degrees 4 and 5 which have positive genus.

\subsection*{Acknowledgements}
The author would like to thank Chiara Damiolini, Darragh Glynn, Rok Gregoric, Rohini Ramadas, Kenny Schefers, and Sara Torelli for helpful conversations related to this project.
Special thanks to Hannah Larson for valuable discussions about computing the 1-dimensional base cases, as well as helpful comments on an earlier version of this paper.
Finally, the author would like to thank their advisor, Sam Payne, for guiding them to an interesting problem and for many productive conversations.
The author was supported in part by NSF DMS–2542134.

\section{Fully-marked admissible covers}

Here, we define the Hurwitz spaces in this paper and describe their relationships to other Hurwitz spaces in the literature.

\begin{definition}[Fully-marked admissible covers]\label{fullmkadm}
  Let $ \underline{d} = (d_1, \dots, d_k)$ be a partition of an integer $d > 0$.
  Let $\underline{g} = (g_1, \dots, g_k)$ be a tuple of nonnegative integers.
  Let $\underline{\mu}= (\mu_1, \dots, \mu_m)$ be a tuple of $m$ partitions of $d$ where each $\mu_i = (\mu_i^1, \dots, \mu_i^{n_i})$.
  A fully-marked admissible cover of type $(\underline{d}, \underline{g}, \underline{\mu})$ is a finite morphism $f\colon C \to D$ satisfying the following:
  \begin{itemize}
    \item The target curve $D$ is an $m$-marked stable curve of genus 0.
    \item The source curve $C$ is a disjoint union of stable curves $C_\ell$ of genus $g_\ell$, marked at each point in the inverse image of the marked locus of $D$.
    The map $C_\ell \to D$ has degree $d_\ell$.
    \item The fibers over the $m$ marked points of $D$ satisfy the ramification profiles given by the $\mu_i$.
    Precisely, if $b_i$ is a marked point on $D$ with ramification profile $\mu_i$, then $f^{-1}(b_i) = \sum_{j = 1}^{n_i} \mu_i^j p_i^j$.
    Consequently, points in $f^{-1}(b_i)$ with the same ramification profile are distinguished by the marking.
    \item The nodes of $C$ are the inverse image of the nodes of $D$.
    At a node, the cover is locally given by $x = z^r $ on both branches of the node for some $r\leq d_\ell$ (known as the kissing condition).
    \item The map $f$ is \'etale outside of the marked locus and nodes of $D$.
  \end{itemize}
\end{definition}

We define families of fully-marked admissible covers as follows.
\begin{definition}
  A family of fully-marked admissible covers of type $(\underline{d}, \underline{g}, \underline{\mu})$ over a scheme $S$ is a commutative diagram

    \begin{center}
      \begin{tikzcd}[column sep=2em, row sep=3em]
      C \arrow[rd] \arrow[rr, "f"] &                                                                 & D \arrow[ld] \\
                                   & S \arrow[lu, "p_i^j", bend left] \arrow[ru, "b_i"', bend right] &
      \end{tikzcd}
    \end{center}
    \noindent where:

    \begin{itemize}
      \item The target curve $D$ is a genus-$0$ curve over $S$ with sections $b_i$.
      \item The source curve $C$ is a disjoint union of stable genus-$g_\ell$ curves $C_\ell$ over $S$, with sections $p_i^j$ satisfying $f^{-1}(b_i) = \sum_{j=1}^{n_i} \mu_i^j p_i^j$.
      The map $C_\ell \to D$ has degree $d_\ell$.
      \item $f$ is \'etale outside of the image of the $b_i$ and the nodes of $D$.
    \end{itemize}
\end{definition}

\begin{definition}[Morphisms of fully-marked admissible covers]\label{covermorphisms}
  A morphism of families of fully-marked admissible covers of type $(\underline{d}, \underline{g}, \underline{\mu})$ from $(f\colon C \to D)$ to $(f'\colon C' \to D')$ over $S$
 is a pair of morphisms $(\phi_{\text{src}}\colon C \to C', \phi_{\text{tgt}}\colon D \to D')$ such that the diagram below commutes:
  \begin{center}
      \begin{tikzcd}[column sep=2em, row sep=5em]
      C \arrow[rr, "f"] \arrow[d, "\phi_{\text{src}}"] &                                                                                                                                           & D \arrow[d, "\phi_{\text{tgt}}"'] \\
      C' \arrow[rd] \arrow[rr, "f'"]                    &                                                                                                                                           & D' \arrow[ld]                      \\
                                                        & S. \arrow[lu, "{p'}_i^j" description, bend left] \arrow[ru, "b'_i"' description, bend right] \arrow[luu, "p_i^j", bend left=60] \arrow[ruu, "b_i"', bend right=60] &
      \end{tikzcd}
  \end{center}
\end{definition}

\begin{definition}[Hurwitz space of fully-marked admissible covers]\label{HurwitzSpace}
  Let $ \underline{d} = (d_1, \dots, d_k)$ be a partition of an integer $d > 0$.
  Let $\underline{g} = (g_1, \dots, g_k)$ be a tuple of nonnegative integers.
  Let $\underline{\mu} = (\mu_1, \dots, \mu_m)$ be a tuple of $m$ partitions of $d$ where each $\mu_i = (\mu_i^1, \dots, \mu_i^{n_i})$.

  The Hurwitz space $\Hbar_{\underline{d},\underline{g}} (\underline{\mu})$ is the moduli stack of isomorphism classes of fully-marked admissible covers of type $(\underline{d}, \underline{g}, \underline{\mu})$.
  There is an open dense substack $\H_{\underline{d},\underline{g}} (\underline{\mu}) \subseteq \Hbar_{\underline{d},\underline{g}} (\underline{\mu})$ parametrizing such admissible covers where $C$ and $D$ are smooth.
  When $\underline{d} = (d)$ and $\underline{g} = (g)$, corresponding to connected covers, we simply write $\Hbar_{d,g}(\underline{\mu})$.
\end{definition}

\begin{remark}
  The Hurwitz space defined in \Cref{HurwitzSpace} is a generalization of the admissible covers compactification of $\H_{d,g}^s$. The geometric picture to have in mind is that when marked fibers approach each other, they ``bubble off''  onto another component (simultaneously on the source and target curves).
  A helpful discussion of this appears at the end of Chapter 3, Section 3 in \cite{harris2006moduli}.
\end{remark}

\subsection{Relation to other Hurwitz spaces}\label{relnotherHurw}
  The Hurwitz space of fully-marked admissible covers $\Hbar_{\underline{d},\underline{g}}(\underline{\mu})$ is the same as the Hurwitz space treated in \cite{CMRTrop}, except that here we fix the target curve to be genus-0.

  In \cite{ACV}, Abramovich, Corti, and Vistoli introduced the space $\overline{\operatorname{Adm}}_{0,m,\underline{d}}$ parametrizing arbitrarily branched covers from possibly disconnected source curves with $m$ marked points on the target genus-0 curve, and showed that it is a smooth stack.
  Here, we use $\underline{d}$ (slightly differing from the notation of \cite{ACV}) to denote the degree of the cover from the different components of the source curve.
  The space $\overline{\operatorname{Adm}}_{0,m,\underline{d}}$ has a covering map from $\Hbar_{\underline{d},\underline{g}}(\underline{\mu})$ given by quotienting by $\prod_{i=1}^m S(\mu_i)$ where $S(\mu_i)$ is the product of symmetric groups corresponding to repeated ramification indices in $\mu_i$.
  For example, the ramification profile $(2,2, 1, 1,1)$ corresponds to the product $S_2 \times S_3$.
  Quotienting by $\prod_{i=1}^m S(\mu_i)$ corresponds to forgetting the ordering in each marked fiber on the source curve.
  In particular, this covering map is \'etale, since its fibers have constant cardinality.
  Therefore $\Hbar_{\underline{d},\underline{g}}(\underline{\mu})$ is also smooth (as a stack).

  We denote the Hurwitz space of connected, fully-marked, simply-branched covers by
   \[\Hbar^s_{d,g}(b,m) := \Hbar_{d,g}((2,1,\dots,1)^b, (1,1,\dots,1)^{m-b}),\]
  \noindent where $b = 2g + 2d - 2$. When $m = b$ , only the branched fibers are marked.
  The Harris-Mumford Hurwitz space introduced in \cite{HM} is the quotient $\Hbar^s_{d,g}(b,b) / \prod_{i=1}^b S_{d-2} $, which we denote by $\overline{\Adm}_{0,b,d}^s$.
  Finally, the Hurwitz space specified in the Picard Rank Conjecture (see \cite{DPPic,MullanePRC}) is the quotient $\overline{\Adm}^s_{0,b,d} / S_b$, which we denote by $\Hbar^s_{d,g}$.

  The Hurwitz spaces $\Hbar_{d,g}(\underline{\mu})$ are the ones discussed in \cite{SvZ} and \cite{FPRelative}.
  We denote the connected admissible covers of \cite{ACV} by $\overline{\operatorname{Adm}}_{0,m,d}$.
  As $m$ varies, the collection of $\overline{\operatorname{Adm}}_{0,m,d} / S_m$ stratifies the Hurwitz spaces $\Hbar_{d,g}$ discussed in \cite{CLlowhurwitz} and in \cite{Zhengtrigonal}.
  In those papers, they work with unmarked, arbitrarily branched covers, so branch points are allowed to collide.

\section{Fully-Marked Monodromy representations}\label{monorep}

The geometric data of a degree-$d$ admissible cover can be recast as the combinatorial data of \textit{monodromy representations}.
This provides an alternative perspective on Hurwitz spaces relying only on combinatorial data, making explicit computations easier.

First, we define families of labeled admissible covers.

\begin{definition}{(Labeled admissible covers).}
  A family of labeled admissible covers of type $(\underline{d}, \underline{g}, \underline{\mu})$ over a scheme $S$ is the data of a fully-marked admissible cover as in \Cref{fullmkadm}, along with a diagram

  \begin{center}
    \begin{tikzcd}[column sep=2em, row sep=3em]
    C \arrow[rd] \arrow[rr, "f"] &                                                                 & D \arrow[ld] \\
                                 & S \arrow[lu, "L_k", bend left] \arrow[ru, "q"', bend right] &
    \end{tikzcd}
  \end{center}
  \noindent where:

  \begin{itemize}
    \item The section $q$ lies in the smooth locus of $D$ and is disjoint from all marked points $b_i$ on $D$.
    \item The $d$ sections $L_k$ on $C$ satisfy $f^{-1}(q) = \sum_{k=1}^d L_k$.
  \end{itemize}
  We call these sections $\{L_k\}$ the \textit{labeling} of $f$, and the tuple $(f, \{L_k\})$ a \textit{labeled admissible cover}.
\end{definition}

To construct a monodromy representation associated to a labeled admissible cover, we fix the following:
\begin{itemize}
  \item Let $S$ be a closed point, and consider a labeled admissible cover $f: C \to D$ over $S$.
  \item Fix an isomorphism $D \cong \P^1$ such that the first three marked points of the marked locus $B \subseteq D$ are sent to $0, 1,$ and $\infty$.
  \item Fix a generating set $\{\gamma_i\}_{i=1}^m$ for $\pi_1(\P^1 \setminus B, q)$ where $\gamma_i$ is a loop based at $q$ winding around $b_i$ once in the clockwise direction.
\end{itemize}

\noindent For any $\gamma_i$ in the generating set for $\pi_1(\P^1 \setminus B, q)$, consider its lift to a path $\widetilde{\gamma}_{i, k}$ in $C$ starting at $L_k$, i.e. $\widetilde{\gamma}_{i,k}(0) = L_k$.
The other endpoint $\widetilde{\gamma}_{i,k}(1)$ of this lift must also lie in $f^{-1}(q)$.
Define $\sigma_i \in S_d$ by $k \mapsto \sigma_i (k)$ where $\widetilde{\gamma}_{i,k}(1) = L_{\sigma_i(k)}$.
Observe that the cycle type of $\sigma_i$ is exactly the ramification profile of $f$ at $b_i$.
For each $\gamma_i$, we can associate such a permutation $\sigma_i$, producing a homomorphism $\pi_1(\P^1 \setminus B, q) \to S_d$.
We package the data of this homomorphism by the ordered tuple $(\sigma_1, \dots, \sigma_m)$.
Finally, we keep track of the marking on $f^{-1}(B)$ by marking the cycles in each of the $\sigma_i$.
For example, in \Cref{fig:markedcover}, the ramification point $b_i$ has profile $2,2,1$.
Given the labeling $\{L_k\}$ on $f^{-1}(q)$, the associated permutation to $b_i$ is $(12)(3)(45)$, where $(12)$ is marked as $p_i^1$, $(3)$ is marked as $p_i^2$ and $(45)$ is marked as $p_i^3$.

\begin{figure}[h]
    \centering
    \includegraphics[scale=0.2]{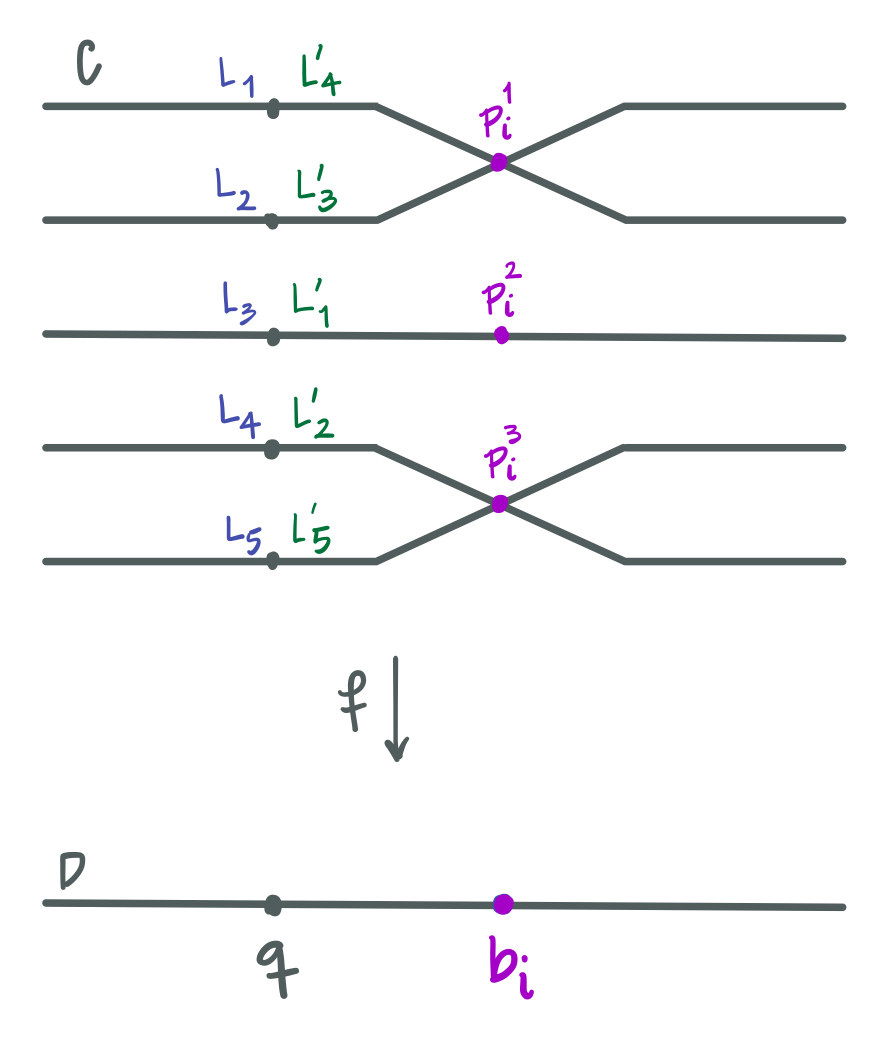}
    \caption{A $1$-dimensional picture of an admissible cover with two labelings $\{L_k\}$ and $\{L'_k\}$ of $f^{-1} (q)$. In terms of $\{L_k\}$, the fiber over $b_i$ is associated to the permutation $(12)(3)(45)$, with markings $p_i^1, p_i^2, p_i^3$ respectively.}
    \label{fig:markedcover}
\end{figure}

Then the \textit{fully-marked monodromy representation} of the labeled admissible cover $(f,L)$ consists of
\begin{enumerate}
  \item a monodromy representation $(\sigma_1, \dots, \sigma_m)$, and
  \item a marking $p_i^j$ on the cycles in each permutation $\sigma_i$.
\end{enumerate}

Two labeled admissible covers $(f \colon C \to \P^1, \{L_k\})$ and $(f' \colon C' \to \P^1, \{L'_k\})$ are isomorphic if there is an isomorphism $\phi_{\operatorname{src}}: C \to C'$ between the fully-marked admissible covers $f$ and $f'$ (see \Cref{covermorphisms}) such that $\phi_{\operatorname{src}} \circ L_k = L'_k$.
Isomorphic labeled admissible covers have identical monodromy representations, by the following argument.
Let $\gamma$ be any generator of $\pi_1(\P^1 \setminus B, q)$.
Let $\sigma$ and $\tau$ be the permutations associated to lifting $\gamma$ via $f$ and $f'$, respectively.
Then, for any $k \in \{1, \dots, d\}$, we have that $\sigma$ is defined by $\widetilde{\gamma}_k(1) = L_{\sigma(k)}$ and $\tau$ is defined by $\widetilde{\gamma}'_k (1) = L'_{\tau(k)}$.
Applying $\phi_{\operatorname{src}}$ to the first equation, we find that:

\begin{align*}
  \widetilde{\gamma}_k (1) &= L_{\sigma(k)}\\
 \phi_{\operatorname{src}}(\widetilde{\gamma}_k (1)) &= \phi_{\operatorname{src}}(L_{\sigma(k)})\\
  \widetilde{\gamma}'_{k}(1) &= L'_{\sigma(k)}.
\end{align*}

\noindent This last step follows since $\phi_{\operatorname{src}}$ is an isomorphism and the lifts $\widetilde{\gamma}_k$ and $\widetilde{\gamma}'_k$ are continuous.
Therefore $\tau(k) = \sigma(k)$.
For this reason, if two labelings of $f$ differ by an automorphism of $f$, then they have identical monodromy representations.

In general, given any two labelings $\{L_k\}$ and $\{L'_k\}$ for $f$, the pairs $(f, \{L_k\})$ and $(f, \{L'_k\})$ have ``simultaneously conjugate'' monodromy representations.
This means that if $(f, \{L_k\}) $ and $ (f, \{L'_k\})$ have the monodromy representations $(\sigma_1, \sigma_2, \dots, \sigma_n) $and $ (\tau_1, \tau_2, \dots, \tau_n)$ respectively, then there exists some $\omega \in S_d$ such that $\sigma_i = \omega \tau_i \omega^{-1}$ for every $i$.
In this case, we have $L_k = L'_{\omega(k)}$.
Furthermore, the markings on $\sigma_i$ and $\tau_i$ have the following relationship:
the marking on the cycle $(a_1 a_2 \cdots a_r)$ of $\sigma_i$ is equal to the marking on the cycle $(\omega(a_1) \omega(a_2) \cdots \omega(a_r))$ of $\tau_i$.
For example, in \Cref{fig:markedcover}, the fiber over $b_i$ is associated to the permutation $(34)(1)(25)$ in terms of the labeling $\{L'_k\}$, with markings $p_i^1, p_i^2, p_i^3$ respectively.

Let the simultaneous conjugacy class $[\sigma_1, \dots, \sigma_m]$ be the set of all tuples $(\tau_1, \dots, \tau_m)$ which are simultaneously conjugate to $(\sigma_1, \dots, \sigma_m)$.
Due to the relation in $\pi_1(\P^1 \setminus B, q)$, the set $\{\sigma_1, \sigma_2, \dots, \sigma_n\}$ must satisfy $\sigma_1 \sigma_2 \cdots \sigma_n = 1$. If the admissible cover $f$ is connected, then $\{\sigma_1, \sigma_2, \dots, \sigma_n\}$ acts transitively on $\{1, 2, \dots, d\}$.
By the previous discussion, the isomorphism class $[f]$ in the Hurwitz space can be associated to $([\sigma_1, \dots, \sigma_m], \{p_i^j\})$.

Conversely, given a tuple $[\sigma_1, \sigma_2, \dots, \sigma_m]$ of permutations from $S_d$ which multiply to the identity and a marking $\{p_i^j\}$ on the cycles in each $\sigma_i$, one can construct a fully-marked admissible cover of $D$.
We give an idea of the construction here.
Pick some representative $(\sigma_1, \sigma_2, \dots, \sigma_m)$.
Consider the polygon $P$ resulting from choosing $q \in \P^1 \setminus B$ and cutting $\P^1$ along simple paths from $q$ to each of the branch points in $B$.
Take $d$ copies of $P$ and denote them by $P^1, P^2, \dots, P^d$.
Denote the sides of polygon $P^k$ corresponding to the path from $q$ to $b_i$ by $u^k_{i}, v^k_{i}$.
We glue these polygons together using $(\sigma_1, \sigma_2, \dots, \sigma_m)$ via the following rule: the side $u^k_{i}$ on $P^k$ should be glued to the side $v^{\sigma_i (k)}_i$ on $P^{\sigma_i(k)}$.
This results in a (possibly disconnected) punctured surface which smoothly covers $\P^1 \setminus B$ with degree $d$.
The Riemann existence theorem guarantees the existence of a unique map of Riemann surfaces which extends this covering over the punctures (see Proposition 1.2 of \cite{FultonHurSch}).
We then add markings in the natural way: over a (filled-in) point $b_i$ in the target curve, we have as many points as there are cycles in $\sigma_i$. Each point corresponds to a cycle, so it inherits the cycle's marking $p_i^j$.
As discussed previously, any simultaneous conjugate of $(\sigma_1, \dots, \sigma_m)$ results in the same cover of Riemann surfaces, since it corresponds to choosing a different labeling.

\subsection{The target map}\label{targetmap}
There is a natural \textit{target map} $T\colon \Hbar_{\underline{d}, \underline{g}}(\underline{\mu}) \to \Mbar_{0,m}$ sending an admissible cover $f \colon C \to D$ to the target curve $D$ marked at $m$ points.
This is a covering map, ramified over the boundary, whose generic fiber at a point $(D, B) \in \M_{0,4}$ is the set of isomorphism classes of fully-marked admissible covers which have marked locus $B \subseteq D$.
Once we fix generators for $\pi_1(D\setminus B, q)$, we have a bijection between the points in this fiber and the set of simultaneous conjugacy classes of fully-marked monodromy representations whose permutations have cycle type equal to the prescribed ramification profiles and multiply to the identity.
Consequently, the sheets of the Hurwitz space above $\Mbar_{0,m}$ can be identified with these fully-marked monodromy representations.
This helps us calculate the monodromy of the target map, which we carry out in \Cref{monotarget}.

\section{The boundary divisors of $\Hbar_{\underline{d}, \underline{g}}(\underline{\mu})$}\label{boundarygeometry}

In this section, we explain why the boundary divisors of $\Hbar_{\underline{d}, \underline{g}}(\underline{\mu})$ are finite group quotients of products of lower-dimensional $\Hbar_{\underline{d'}, \underline{g'}}(\underline{\mu'})$.

First, we relate the boundary of the Hurwitz space with the boundary of $\Mbar_{0,m}$.
The target map sends boundary divisors of $\Hbar_{\underline{d}, \underline{g}}(\underline{\mu})$ to the boundary divisors $\{\Delta_j\}$ of $\Mbar_{0,m}$.
The divisor $\Delta_j$ parametrizes genus-0 curves $D$ consisting of two components connected by a single node, with $j$ marked points on one component $D_1$ and $b-j$ marked points on the other component $D_2$ (see the lower graph in \Cref{fig:boundarydivisors} for the dual graph of $\Delta_j$).
There are multiple isomorphic components in $\Delta_j$, corresponding to the different labelings of the marked points.
Above $\Delta_j$, the boundary divisors of $\Hbar_{\underline{d}, \underline{g}}(\underline{\mu})$ parametrize admissible covers of $D$ where the fibers over the marked points have the ramification profiles $\mu_i$.
On each component of $D$, the permutations associated to the marked points and the node must multiply to the identity.
Therefore, the ramification profile over the node is constrained by the ramification profiles of the fibers over the marked points on either component.
We can associate a dual graph to an admissible cover by taking the usual dual graph construction of a marked curve for both the source and target curves (see details in Section 3.2 of \cite{CMRTrop}).
Additionally, we label the edges and half-edges in the source graph with the ramification index.
In this way, every boundary divisor in the Hurwitz space has a dual graph of admissible covers $\Gamma$, and we refer to such a boundary divisor by $\Delta_\Gamma$.
There may be multiple $\Delta_\Gamma$ above each component of $\Delta_j$, as in \Cref{fig:boundarydivisors}.

\begin{figure}[h]
    \centering
    \includegraphics[scale=0.25]{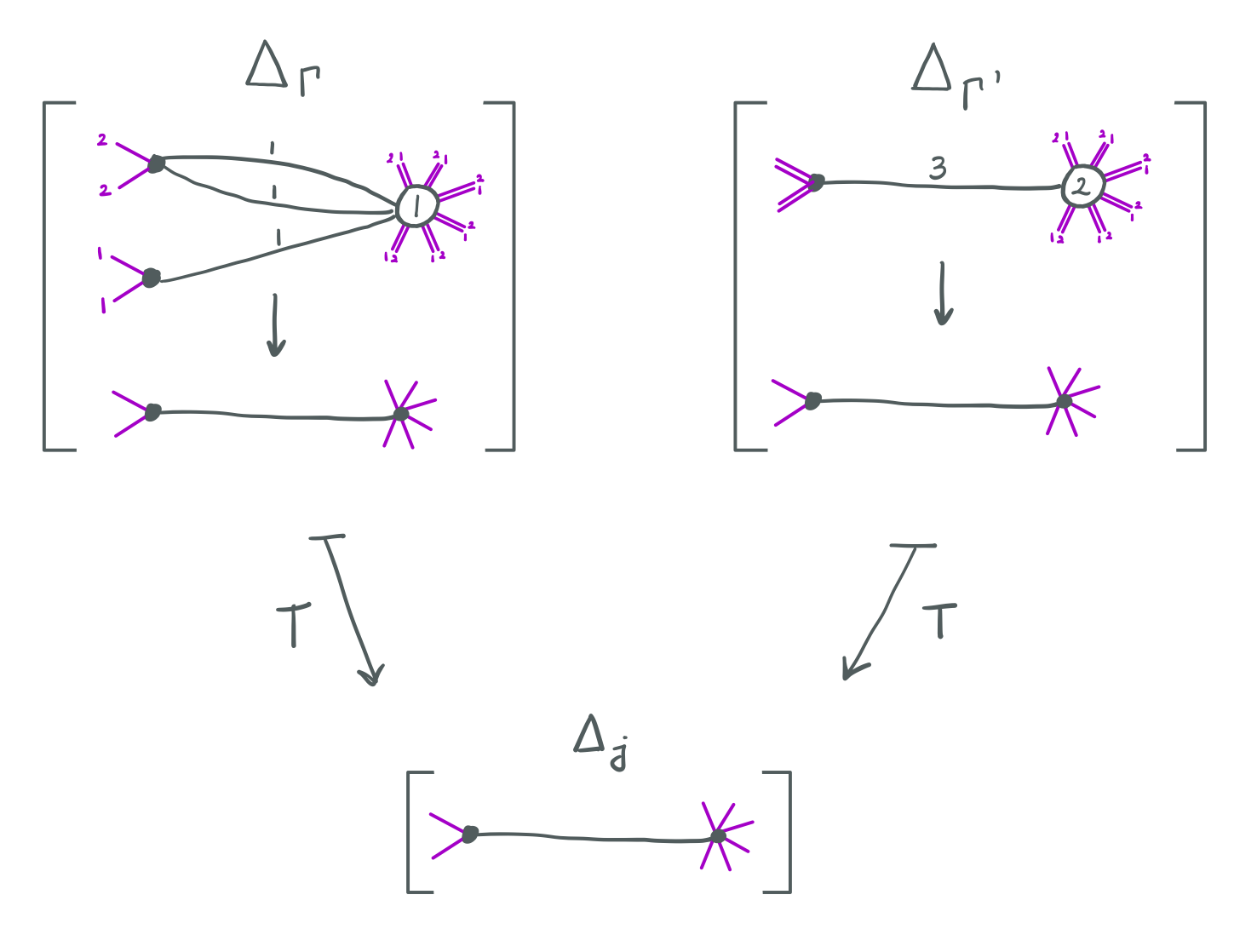}
    \caption{The dual graphs of two boundary divisors of $\Hbar_{3,2}((2,1)^8)$ lying above a single component of $\Delta_j$ for $j=2$ (with marking data suppressed to minimize cluttering the picture). The two branch points which have bubbled off to the left vertex have associated transpositions in $S_3$. If these transpositions are identical, then the resulting boundary divisor is $\Delta_\Gamma$. If these transpositions have one element in common, then the resulting boundary divisor is $\Delta_{\Gamma'}.$}
    \label{fig:boundarydivisors}
\end{figure}
There are natural gluing maps
\begin{equation}\label{gluingmap}
  \begin{aligned}
    \xi_\Gamma\colon \Hbar_{\underline{d_1}, \underline{g_1}}(\underline{\mu_1}) \times \Hbar_{\underline{d_2}, \underline{g_2}}(\underline{\mu_2}) & \xlongrightarrow{} \Hbar_{\underline{d_3}, \underline{g_3}}(\underline{\mu_3})
  \end{aligned}
\end{equation}
\noindent associated to each $\Delta_\Gamma$ such that $\underline{d_1}, \underline{d_2}$ are partitions of $d_3$ and $(m_1 - 1) + (m_2 -1) = m_3$, where $m_i$ is the length of $\underline{\mu_i}$.
The gluing map is given by specifying that marked points should be identified if they correspond to an edge in the dual graph.
Then $\Delta_\Gamma$ is the quotient of the image of $\xi_\Gamma$ by the group of graph automorphisms $\Aut(\Gamma)$.
For example, the divisor $\Delta_\Gamma$ in \Cref{fig:boundarydivisors} is the quotient of the image of the gluing map $\xi_\Gamma \colon \Hbar_{(2,1),(0,0)}((2,1)^2, (1,1,1)) \times \Hbar_{3,1}((2,1)^6, (1,1,1)) \to \Hbar_{3,2}((2,1)^8)$ by $S_2$, since there are $\vert S_2 \vert$-many ways to choose the labelings on the preimage of the glued fiber.

\begin{remark}
Examining the analogous gluing maps to the boundary of $\Hbar^s_{d,g}$ illustrates how $\Hbar^s_{d,g}$ is not stratified by lower-dimensional $\Hbar^s_{d',g'}$, but rather by Hurwitz spaces of possibly disconnected, arbitrarily-branched admissible covers with some number of marked fibers.
For example, the Hurwitz spaces in the domain of the gluing maps to the boundary divisors of $\Hbar^s_{d,g}$ parametrize covers with one marked fiber.
As the codimension of the boundary stratum increases, more marked fibers are necessary to define the gluing maps.
Consequently, any inductive proof of odd cohomology vanishing for $\Hbar^s_{d,g}$ in the style of \Cref{main} necessarily involves studying the more general spaces $\Hbar_{\underline{d},\underline{g}}(\underline{\mu})$.
\end{remark}

\section{Vanishing of $H^1$}
In this section, we prove that the cohomology of $\Hbar_{\underline{d}, \underline{g}}(\underline{\mu})$ injects into the cohomology of the normalization of the boundary whenever it injects into the cohomology of the boundary (\Cref{injlemma}).
We use this result to prove that the odd cohomology of $\Hbar_{\underline{d}, \underline{g}}(\underline{\mu})$ vanishes for all $\underline{d}, \underline{g},$ and $\underline{\mu}$ if certain base cases have this property (\Cref{indodd}).
We apply this argument for $H^1$ and explain how to compute the base cases.

\subsection{Injectivity of pullback map on cohomology}\label{injectivitysection}
Recall that the target map $T\colon \calH_{\underline{d}, \underline{g}}(\underline{\mu}) \to \M_{0,m}$ is a finite cover (see \Cref{targetmap}).
As $\M_{0,m}$ is affine of dimension $m-3$, it follows that $\H_{\underline{d}, \underline{g}}(\underline{\mu})$ is also affine of dimension $m-3$.
Therefore the homology groups $H_k(\calH_{\underline{d}, \underline{g}}(\underline{\mu}))$ vanish for $k>m - 3$, and by Poincar\'e duality, the compactly supported cohomology $H_c^k(\calH_{\underline{d}, \underline{g}}(\underline{\mu}))$ vanishes for $k<m - 3$.
The long exact sequence for compactly supported cohomology is as follows:

\[
\cdots \to H_c^k (\calH_{\underline{d}, \underline{g}}(\underline{\mu})) \to H^k (\calHbar_{\underline{d}, \underline{g}}(\underline{\mu})) \to H^k(\del \calHbar_{\underline{d}, \underline{g}}(\underline{\mu})) \to H_c^{k+1}(\calH_{\underline{d}, \underline{g}}(\underline{\mu})) \to \cdots
\]

Hence the map  $H^k(\calHbar_{\underline{d}, \underline{g}}(\underline{\mu})) \to H^k (\del \calHbar_{\underline{d}, \underline{g}}(\underline{\mu}))$ is injective when $k \leq m -4 $.
Now we show that $H^k(\Hbar_{\underline{d}, \underline{g}}(\underline{\mu}))$ also injects into the cohomology of the normalization of the boundary in this range.

\begin{lemma}[Injectivity lemma]\label{injlemma}
  Let $\operatorname{\xi_\Gamma}\colon \Hbar_\Gamma \to \partial\Hbar_{\underline{d}, \underline{g}}(\underline{\mu})$ be the gluing map associated to the boundary divisor parametrizing admissible covers with dual graph $\Gamma$, where each $\Hbar_\Gamma$ is a product $\Hbar_{\underline{d_1}, \underline{g_1}}(\underline{\mu_1}) \times \Hbar_{\underline{d_2}, \underline{g_2}}(\underline{\mu_2})$.
  If $\rho\colon H^k(\calHbar_{\underline{d}, \underline{g}}(\underline{\mu})) \to H^k(\partial \calHbar_{\underline{d}, \underline{g}}(\underline{\mu}))$ is injective, then the map $H^k(\calHbar_{\underline{d}, \underline{g}}(\underline{\mu})) \to \oplus_\Gamma H^k(\Hbar_\Gamma)$ is injective, where the direct sum is taken over all boundary divisors.
\end{lemma}

\begin{proof}
  Let $u\colon \amalg_\Gamma \Hbar_\Gamma \to \partial\calHbar_{\underline{d}, \underline{g}}(\underline{\mu})$ be the disjoint union of the gluing maps $\xi_\Gamma$.
  It is a proper and surjective map.
  Similarly to the discussion for $\calMbar_{g,n}$ in \cite[Prop. 2.7]{ACcoh}, the weight $k$ quotient \[H^k(\partial \calHbar_{\underline{d}, \underline{g}}(\underline{\mu}))/W_{k-1}H^k (\partial \calHbar_{\underline{d}, \underline{g}}(\underline{\mu}))\] is the image of $H^k(\partial \calHbar_{\underline{d}, \underline{g}}(\underline{\mu}))$ under $u^*$.
  Therefore we can factor $u^*$ as in the following diagram:

  \begin{center}
    \begin{tikzcd}[row sep=4em, column sep=1em]
  {H^k(\calHbar_{\underline{d}, \underline{g}}(\underline{\mu}))} \arrow[rr, "\rho"]  &  & {H^k(\partial \calHbar_{\underline{d}, \underline{g}}(\underline{\mu}))} \arrow[rr, "u^*"] \arrow[rrd, "f^*", two heads] &  & H^k (\amalg_\Gamma \Hbar_\Gamma)                                                                  \\
  &  &                                                                &  & {H^k(\partial \calHbar_{\underline{d}, \underline{g}}(\underline{\mu}))/W_{k-1}H^k (\partial \calHbar_{\underline{d}, \underline{g}}(\underline{\mu}))}. \arrow[u]
  \end{tikzcd}
  \end{center}

To show that $u^* \circ \rho$ is injective, it suffices to show that $f^*$ is injective on the image of $\rho$.
As $\rho$ is a restriction map, it is a morphism of mixed Hodge structures and must respect the weight filtration.
Furthermore, the cohomology $H^k(\Hbar_{\underline{d}, \underline{g}}(\underline{\mu}))$ is pure weight $k$, since $\Hbar_{\underline{d}, \underline{g}}(\underline{\mu})$ is smooth.
Hence the intersection $\rho(H^k(\calHbar_{\underline{d}, \underline{g}}(\underline{\mu}))) \cap W_{k-1}H^k(\partial \calHbar_{\underline{d}, \underline{g}}(\underline{\mu}))$ is zero.
\end{proof}

\subsection{Inductive vanishing of odd cohomology}

\begin{proposition}\label{indodd}
Let $k$ be an odd integer.
Suppose that $H^q(\Hbar_{\underline{d},\underline{g}}(\underline{\mu})) = 0$ for all odd $q \leq k$ and parameters $(\underline{d},\underline{g},m,\underline{\mu})$ such that $q > m-4$.
Then $H^q(\Hbar_{\underline{d},\underline{g}}(\underline{\mu})) = 0$ for all odd $q \leq k$ and all parameters $(\underline{d},\underline{g},m,\underline{\mu})$.
\end{proposition}
\begin{proof}
We induct on $k$. Assume the claim is true for all $q < k$ and all parameters $ (\underline{d}, \underline{g},m,\underline{\mu})$.
We need to show that $H^k(\calHbar_{\underline{d}, \underline{g}}(\underline{\mu})) = 0$ for all $ (\underline{d}, \underline{g},m,\underline{\mu})$ with $k \leq m-4$.

\Cref{injlemma} holds in this range, so $H^k(\calHbar_{\underline{d}, \underline{g}}(\underline{\mu})) \rightarrow \oplus_\Gamma H^k(\Hbar_\Gamma)$ is injective.
Each space $\Hbar_\Gamma$ is a product $ \Hbar_{\underline{d_1}, \underline{g_1}}(\underline{\mu_1}) \times \Hbar_{\underline{d_2}, \underline{g_2}}(\underline{\mu_2})$ where $\underline{d_1}, \underline{d_2}$ are partitions of $d$ and $(m_1- 1) + (m_2-1) = m$.
Crucially, the number of marked fibers $m_i$ on each factor of the domain is strictly less than $m$.
By the K\"{u}nneth formula, the cohomology groups decompose as \[H^k(\Hbar_\Gamma) = \bigoplus_{p+q = k} H^p( \calHbar_{\underline{d_1}, \underline{g_1}}(\underline{\mu_1})) \otimes H^q(\calHbar_{\underline{d_2}, \underline{g_2}}(\underline{\mu_2})).\]
As $k$ is odd, there must be a factor with odd cohomological degree in each summand.
These factors with odd degree strictly less than $k$ vanish via the induction hypothesis, and therefore $H^k(\Hbar_\Gamma)$ is a direct sum of tensor products where $p$ and $q$ are either $k$ or 0.

If $k > m_1 - 4$, then $H^k(\Hbar_{\underline{d_1}, \underline{g_1}}(\underline{\mu_1})) = 0$.
Otherwise $k \leq m_1 - 4$ and one applies the previous argument again to get an injection into a direct sum of tensor products of $H^k(\Hbar_{\underline{d_3}, \underline{g_3}}(\underline{\mu_3}))$ and $H^0(  \Hbar_{\underline{d_4}, \underline{g_4}}(\underline{\mu_4}))$, with $m_3, m_4 < m_1$.
The same reasoning holds for $H^k(\Hbar_{\underline{d_2}, \underline{g_2}}(\underline{\mu_2})$.
By iterating this argument, eventually $k$ becomes larger than $m_3 - 4, m_4 - 4,$ etc., and then $H^k(\Hbar_\Gamma) = 0$.
Therefore
$H^k(\calHbar_{\underline{d}, \underline{g}}(\underline{\mu})) = 0$.
\end{proof}

\begin{remark}\label{futurework}
  We use \Cref{indodd} to prove $H^1$ vanishing for $\calHbar_{\underline{3}, \underline{g}}(\underline{\mu})$ in this paper.
  As the analogous proposition for $\Mbar_{g,n}$ was used to prove vanishing of odd cohomology in higher degree, it may be possible to do the same for $\calHbar_{\underline{3}, \underline{g}}(\underline{\mu})$.
  Since the number of base cases grows with the cohomological degree, a successful proof would likely involve calculating Euler characteristics and a robust understanding of the combinatorics of the boundary strata of $\calHbar_{\underline{3}, \underline{g}}(\underline{\mu})$.
\end{remark}

\subsection{Computing $H^1$ for degree-$3$ covers}\label{h1proof}
As discussed in \Cref{injectivitysection}, the map $H^k(\calHbar_{\underline{d}, \underline{g}}(\underline{\mu})) \to \oplus_\Gamma H^k(\Hbar_\Gamma)$ is injective for all $k \leq m - 4$.
In particular, $H^1(\calHbar_{\underline{d}, \underline{g}}(\underline{\mu})) \to \oplus_\Gamma H^1(\Hbar_\Gamma)$ is injective whenever $1 \leq m-4$, i.e. whenever the target curve has $5$ or more marked points.
To show that $H^1(\calHbar_{\underline{d},\underline{g}}(\underline{\mu})) = 0 $ using \Cref{indodd} for $d = 3$, we must show that $H^1(\calHbar_{\underline{d},\underline{g}}(\underline{\mu})) = 0$ for every Hurwitz space of admissible covers with $3$ or $4$ marked fibers.
As Hurwitz spaces with $3$ marked fibers are 0-dimensional, we are left with the 1-dimensional spaces.
In this section, we describe how to compute the genera of these 1-dimensional Hurwitz spaces using the monodromy of the covering map to $\calMbar_{0,4}$ and the Riemann-Hurwitz formula.
In \Cref{calculationdata}, we record that the base cases for degree-$2$ and degree-$3$ covers all have genus 0.

\subsubsection{Hurwitz spaces are themselves branched covers}\label{bigmonodromy}

The target map $T\colon \Hbar_{\underline{d},\underline{g}}(\underline{\mu}) \to \Mbar_{0,m}$ is a finite covering map ramified over the boundary.
In the case where the Hurwitz space is 1-dimensional, the Riemann-Hurwitz formula applied to $T$ determines the genus of the Hurwitz space.

There are two steps in this Riemann-Hurwitz calculation: first, determine the sheets of the Hurwitz space lying over $\M_{0,m}$; second, calculate the monodromy of $T$ over the boundary $\partial\Mbar_{0,m}$.
To do this, we recast the geometric data of a point in Hurwitz space into the algebraic data of fully-marked monodromy representations, as described in \Cref{monorep}.
A word of caution: the word ``monodromy'' is used in two situations in this paper, since each point of Hurwitz space is an admissible cover with monodromy (and has a fully-marked monodromy representation), and the target map $T$ itself has monodromy.
We will endeavor to make it clear to which situation we refer.

\subsubsection{Monodromy of the target map}\label{monotarget}
Although it is possible to have the following discussion for arbitrary $m$, we restrict to $m = 4$ as we need only perform these monodromy calculations for the $1$-dimensional Hurwitz spaces.
To compute the monodromy of $T$, we must analyze the sheets of $\H_{\underline{d},\underline{g}}(\mu_1, \mu_2, \mu_3, \mu_4)$ lying over $\M_{0,4}$ and how they come together over $\partial\Mbar_{0,4}$.

For any point $[f\colon C \to D, B] \in \H_{\underline{d},\underline{g}}(\mu_1, \mu_2, \mu_3, \mu_4)$, we identify $D \cong \P^1$ and $B = \{0, 1, \infty, \lambda\}$.
Let $q \in \P^1 \setminus B$ and fix a generating set for $\pi_1(\P^1 \setminus B, q)$ consisting of four simple loops $\gamma_1, \gamma_2, \gamma_3,$ and $ \gamma_4$ based at $q$ that wind once counterclockwise around $0, 1, \infty,$ and $\lambda$ respectively in $\P^1$.
We refer to this generating set as the ``standard'' generating set, depicted in \Cref{fig:standardlollipop}.
\begin{figure}[h]
    \centering
    \includegraphics[scale=0.15]{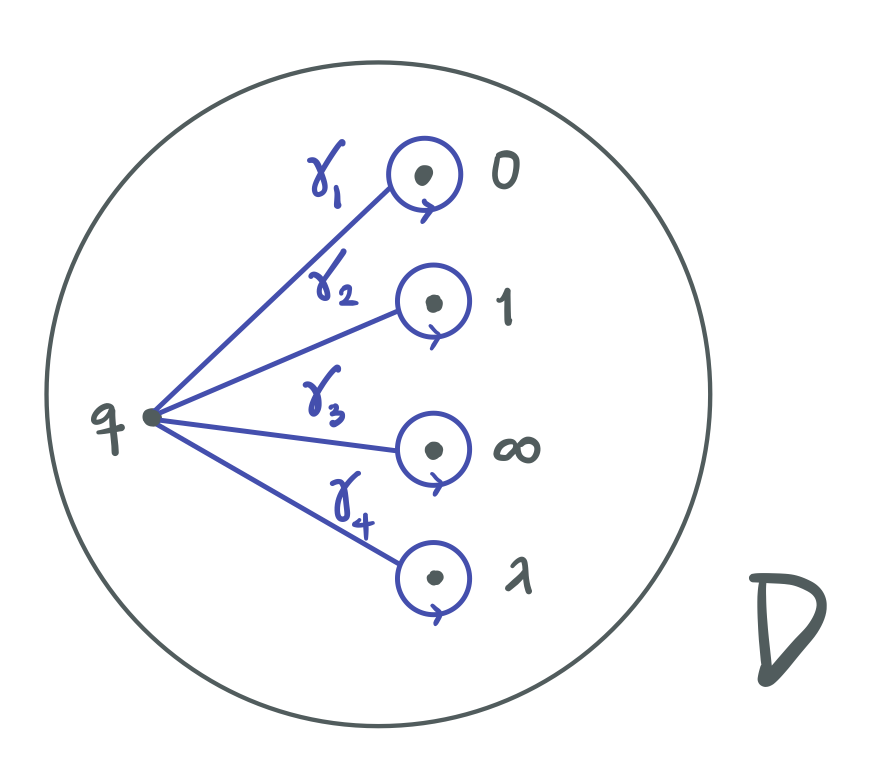}
    \caption{The standard generating set for $\pi_1(\P^1 \setminus B, q)$.}
    \label{fig:standardlollipop}
\end{figure}

Having fixed the standard generating set for $\pi_1(\P^1 \setminus B, q)$, the sheets of $\H_{\underline{d},\underline{g}}(\mu_1, \mu_2, \mu_3, \mu_4)$ are identified with the equivalence classes of fully-marked monodromy representations whose permutations have marked cycles, satisfy the specified ramification profiles, and multiply to the identity.
Furthermore, identifying $D$ with $\P^1$ allows us to identify $\Mbar_{0,4}$ with $\P^1$ in the usual manner:
\begin{align*}
  \Mbar_{0,4} &\xlongrightarrow{\sim} \P^1 \\
  [D, 0, 1, \infty, \lambda] &\longmapsto \lambda.
\end{align*}

\noindent To disambiguate these two copies of $\P^1$, we consistently use the names $D$ and $\Mbar_{0,4}$, respectively.

To determine how the sheets of the Hurwitz space come together over $\partial\Mbar_{0,4} = \{0,1,\infty\}$, we analyze what happens when we lift a loop in $\M_{0,4}$ winding around any of the points in the boundary.
Consider a small open neighborhood $N$ of $\infty \in \Mbar_{0,4}$ and let $\lambda \in \M_{0,4}$ lie inside this neighborhood.
Let $\nu$ be a loop based at $\lambda$ which wraps once around $\infty$
(see \Cref{fig:loopinftyandnewgen}).
\begin{figure}[h]
    \centering
    \includegraphics[scale=0.21]{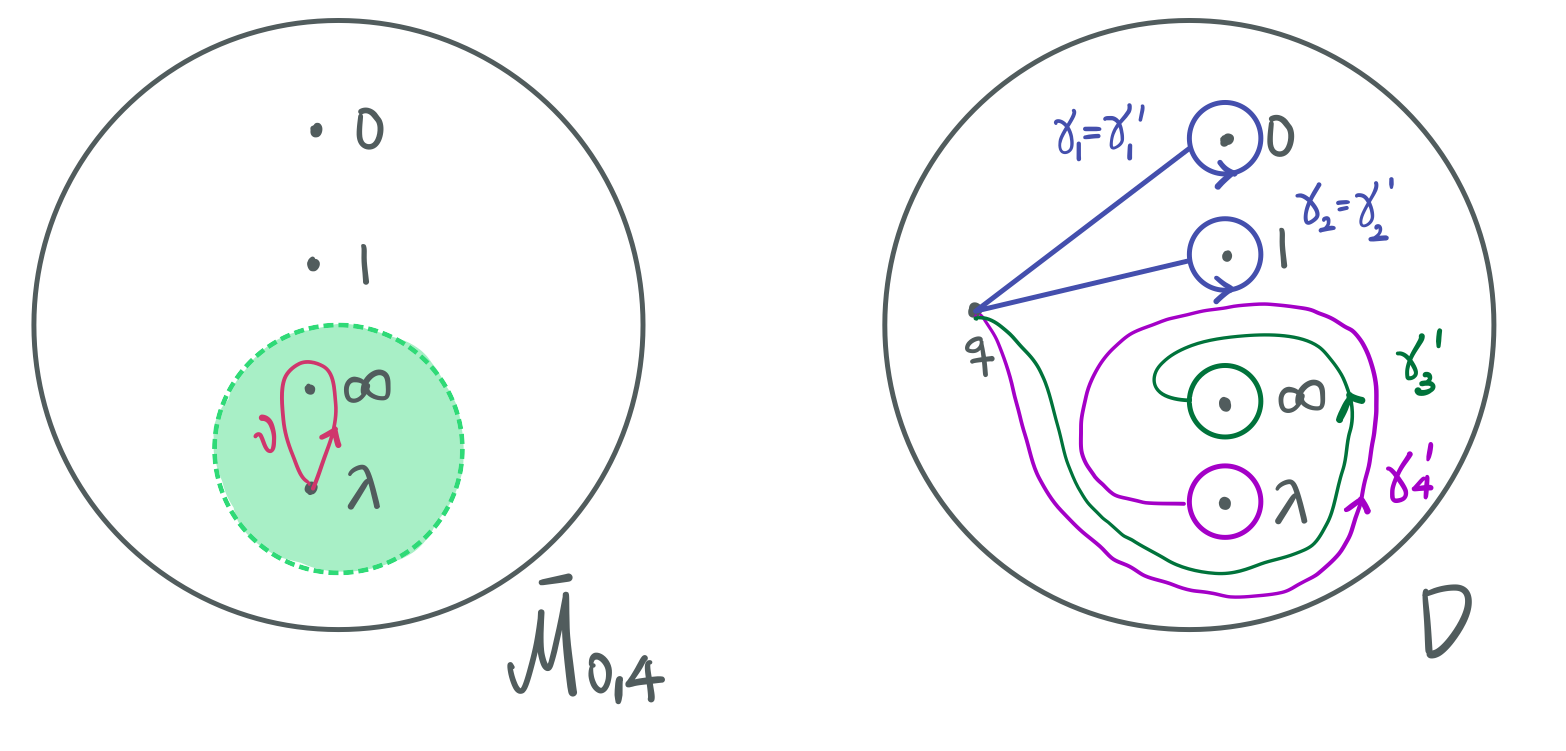}
    \caption{A loop in $\Mbar_{0,4}$ based at $\lambda$ and winding around $\infty$, which perturbs the genrators of $\pi_1(D\setminus B, q)$ into a new generating set on the right. This is a two-dimensional picture.}
    \label{fig:loopinftyandnewgen}
\end{figure}
The sheets of $\Hbar_{\underline{d},\underline{g}}(\mu_1, \mu_2, \mu_3, \mu_4)$ lying above $N$ correspond to the fiber of $\lambda \in \Mbar_{0,4}$, and therefore are given by all simultaneous conjugacy classes of permutations $[\sigma_1, \sigma_2, \sigma_3, \sigma_4]$ multiplying to the identity which have cycle types $\mu_i$
and marked cycles $p_i^j$ for every $i \in \{1,2,3,4\}$
(see \Cref{fig:localtarget}).
\begin{figure}[h]
    \centering
    \includegraphics[scale=0.18]{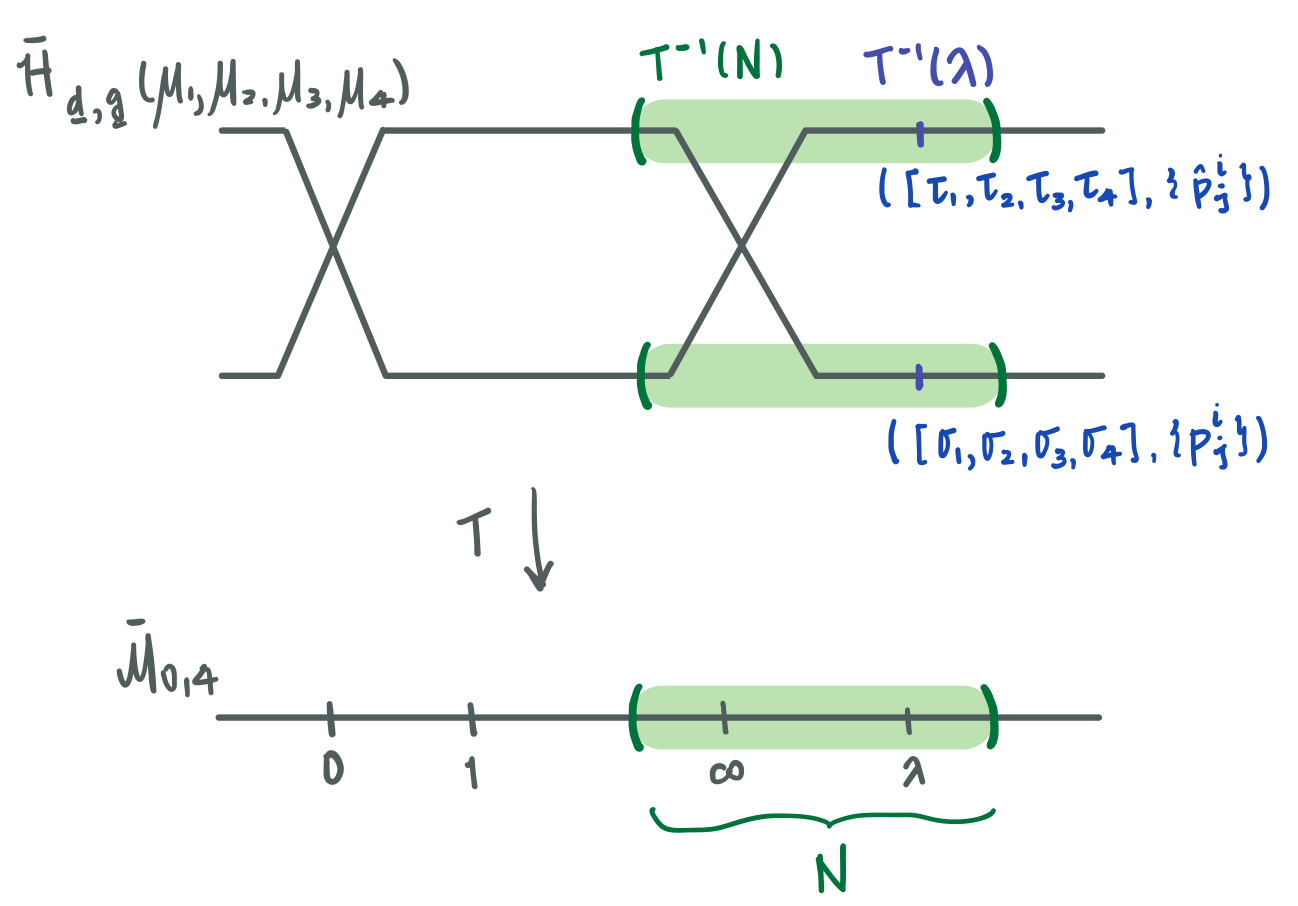}
    \caption{The target map locally near $\lambda$. Two sheets of the Hurwitz space are given by fully-marked monodromy representations. This is a one-dimensional picture.}
    \label{fig:localtarget}
\end{figure}

Consider a point in such a sheet, given by the fully-marked monodromy representation $([\sigma_1, \sigma_2, \sigma_3, \sigma_4], \{p^i_j\})$ in the preimage $T^{-1}(\lambda)$, and lift $\nu$ to a path $\widetilde{\nu}$ in $\Hbar_{\underline{d},\underline{g}}(\mu_1, \mu_2, \mu_3, \mu_4)$ starting at $([\sigma_1, \sigma_2, \sigma_3, \sigma_4], \{p^i_j\})$.
The path $\widetilde{\nu}$ corresponds to the point $\lambda \in D$ moving around $\infty \in D$, which perturbs the loops $\gamma_3$ to $\gamma'_3$ and $\gamma_4$ to $\gamma'_4$ as in \Cref{fig:loopinftyandnewgen}.
The endpoint $\widetilde{\nu}(1)$ is the monodromy representation $([\sigma_1, \sigma_2, \sigma_3, \sigma_4], \{p^i_j\})$ with respect to this new generating set $\{\gamma'_1, \gamma'_2, \gamma'_3, \gamma'_4\}$, where $\gamma'_1 = \gamma_1$ and $\gamma'_2 = \gamma_2$.
To see on which sheet $\widetilde{\nu}(1)$ lies, we calculate the monodromy representation of $\widetilde{\nu}(1)$ with respect to the standard generating set.

Suppose $[\tau_1, \tau_2, \tau_3, \tau_4]_{\text{std}}$ is the monodromy representation of $\widetilde{\nu}(1)$ with respect to the standard generating set and $[\sigma_1, \sigma_2, \sigma_3, \sigma_4]_{\text{new}}$ is the monodromy representation of $\widetilde{\nu}(1)$ with respect to the new generating set.
These two monodromy representations correspond to the same admissible cover, so $\gamma'_i \mapsto \sigma_i$ and $\gamma_i \mapsto \tau_i$ are the same group homomorphism $\pi_1(D \setminus B, q) \longrightarrow S_d$ (refer to \Cref{monorep} for details about this homomorphism).
The new generating set can be written in terms of the standard generating set as:
\vspace{-6mm}
\begin{align*}
  \gamma'_1 &= \gamma_1\\
  \gamma'_2 &= \gamma_2\\
  \gamma'_3 &= \gamma_4 \gamma_3 \gamma^{-1}_4\\
  \gamma'_4 &= \gamma_4 \gamma_3 \gamma_4 \gamma^{-1}_3 \gamma^{-1}_4.
\end{align*}
\vspace{-6mm}

\noindent Therefore the monodromy representation with respect to the standard generating set must be:

\vspace{-6mm}
\begin{equation}\label{formulainfty}
  \begin{aligned}
    \tau_1 &= \sigma_1\\
    \tau_2 &= \sigma_2\\
    \tau_3 &= \sigma_3 \sigma_4 \sigma_3 \sigma^{-1}_4\sigma^{-1}_3\\
    \tau_4 &= \sigma_3 \sigma_4 \sigma^{-1}_3,
  \end{aligned}
\end{equation}

\vspace{-6mm}
\noindent Observe that each $\tau_i$ has the form $\nu \sigma_i \nu^{-1}$.
The marking on the cycle $(a_1 a_2 \cdots a_r)$ of $\tau_i$ must be equal to the marking on the cycle $(\nu^{-1}(a_1) \nu^{-1}(a_2) \cdots \nu^{-1}(a_r))$.
Hence $\widetilde{\nu}(1)$ lies on the sheet of $\Hbar_{\underline{d},\underline{g}}(\mu_1, \mu_2, \mu_3, \mu_4)$ given by the monodromy representation $[\sigma_1, \sigma_2, \sigma_3 \sigma_4 \sigma_3 \sigma^{-1}_4\sigma^{-1}_3, \sigma_3 \sigma_4 \sigma^{-1}_3]_\textrm{std}$ and marking specified by the previous statement.

\begin{figure}[h]
    \centering
    \includegraphics[scale=0.2]{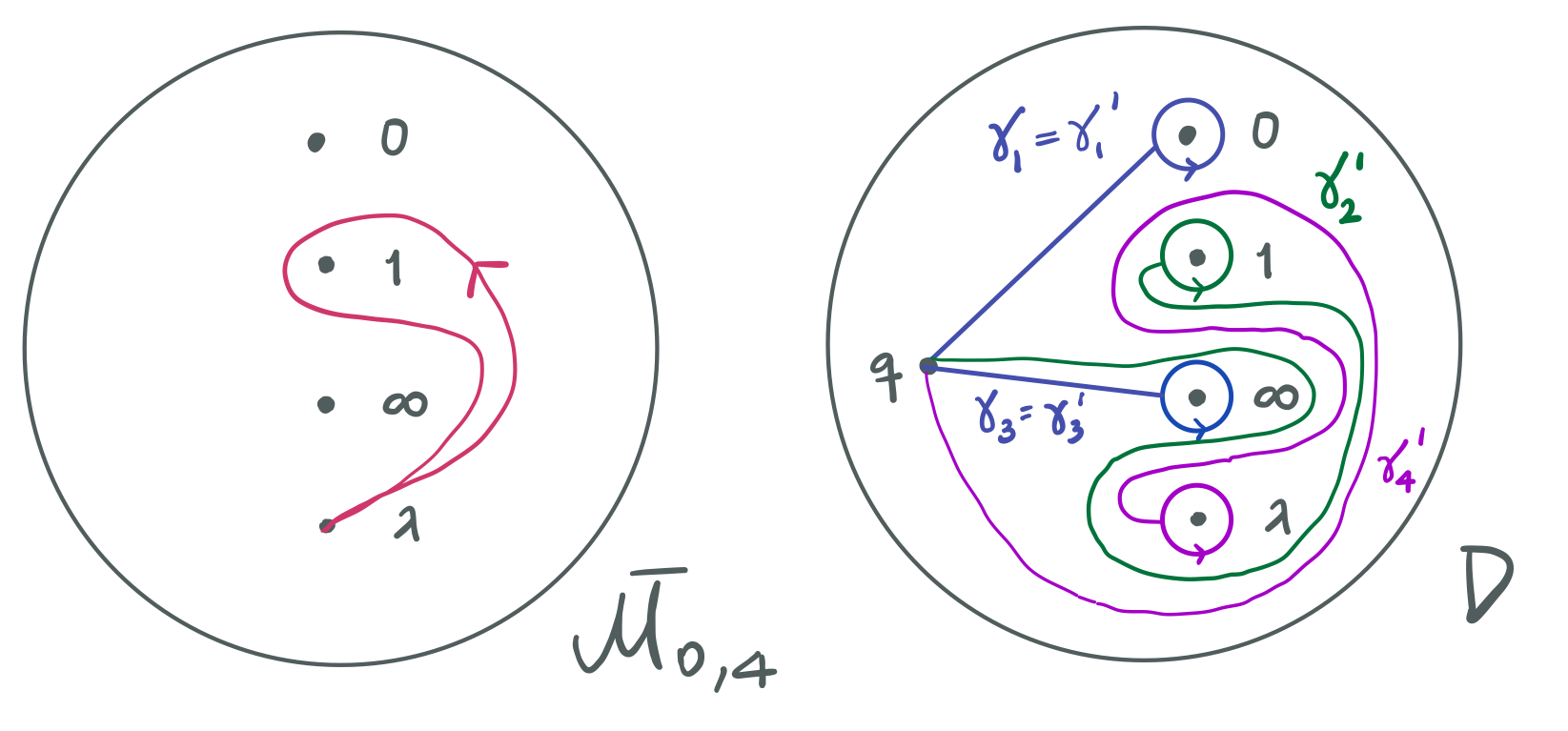}
    \caption{A loop in $\Mbar_{0,4}$ based at $\lambda$ and winding around $1$, which perturbs the generators of $\pi_1(D \setminus B, q)$ into a new generating set on the right. This is a two-dimensional picture.}
    \label{fig:loopone}
\end{figure}
\pagebreak
A similar calculation can be done for a loop $\nu$ based at $\lambda$ and winding once around $1$ (see \Cref{fig:loopone}). The resulting monodromy representation with respect to the standard generating set is:

\begin{equation}\label{formulaone}
\begin{aligned}
  \tau_1 &= \sigma_1\\
  \tau_2 &= \sigma_2 \sigma_3 \sigma_4 \sigma_3^{-1}\sigma_2 \sigma_3 \sigma_4^{-1} \sigma_3^{-1} \sigma_2^{-1}\\
  \tau_3 &= \sigma_3\\
  \tau_4 &= \sigma_3^{-1}\sigma_2 \sigma_3 \sigma_4 \sigma_3^{-1}\sigma_2^{-1}\sigma_3.
\end{aligned}
\end{equation}

\begin{figure}[h]
    \centering
    \includegraphics[scale=0.2]{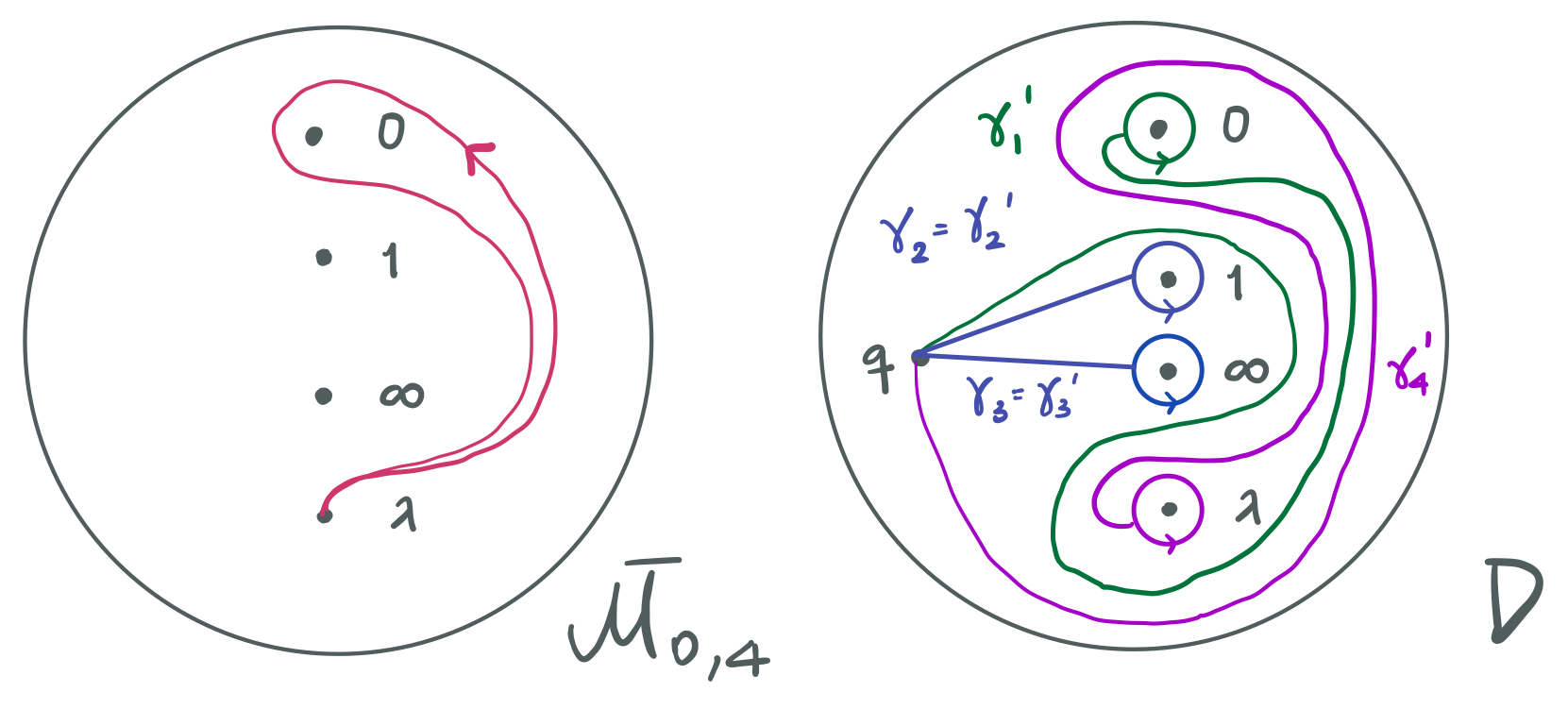}
    \caption{A loop in $\Mbar_{0,4}$ based at $\lambda$ and winding around $0$, which perturbs the generators of $\pi_1(D \setminus B, q)$ into a new generating set on the right. This is a two-dimensional picture.}
    \label{fig:loopzero}
\end{figure}

Finally, for a loop based at $\lambda $ and winding once around $0$ (see \Cref{fig:loopzero}), the resulting monodromy representation with respect to the standard generating set is:

\begin{equation}\label{formulazero}
  \begin{aligned}
    \tau_1 &= \sigma_1 \sigma_2 \sigma_3 \sigma_4 \sigma_3^{-1} \sigma_2^{-1} \sigma_1 \sigma_2 \sigma_3 \sigma_4^{-1} \sigma_3^{-1} \sigma_2^{-1} \sigma_1^{-1}\\
    \tau_2 &= \sigma_2\\
    \tau_3 &= \sigma_3\\
    \tau_4 &= \sigma_3^{-1}\sigma_2^{-1}\sigma_1 \sigma_2 \sigma_3 \sigma_4 \sigma_3^{-1} \sigma_2^{-1} \sigma_1^{-1} \sigma_2 \sigma_3.
  \end{aligned}
\end{equation}
\subsubsection{Python program for computing monodromy of target map.}\label{python}
By calculating the endpoints of lifts of loops to $\Hbar_{\underline{d},\underline{g}}(\mu_1, \mu_2, \mu_3, \mu_4)$, we can compute the monodromy of $T$.
We implemented this calculation in Python.
The code (with examples) for this program is provided at \url{https://github.com/amyqhli/1dim-hurwitz-space-monodromy-calculator}. We briefly describe the algorithm here.

Given a Hurwitz space $\Hbar_{\underline{d},\underline{g}}(\mu_1, \mu_2, \mu_3, \mu_4)$, we first set \texttt{degree} equal to $d$ and produce a list of tuples \texttt{candidates} $ = \{(\sigma_1, \sigma_2, \sigma_3, \sigma_4)\} \subseteq (S_d)^4$ where $\sigma_i$ has ramification profile $\mu_i$.
Second, we multiply $\sigma_1 \sigma_2 \sigma_3 \sigma_4$ (the order of multiplication matters) to check if each tuple of permutations multiplies to the identity using the function \texttt{CheckIdentity(degree, candidates)}.
This function outputs a list called \texttt{successes}.
We add all the possible markings to each monodromy representation in \texttt{successes}, and then filter out tuples which are simultaneous conjugates of other tuples using the function \texttt{marked\_sc\_eqclass(degree, marked\_successes)}.
We are left with a list \texttt{equivs} of simultaneous conjugacy classes of fully-marked monodromy representations.
We create a dictionary \texttt{HurwSpace} which enumerates each tuple in \texttt{equivs} as \texttt{"Sheet\textbraceleft i\textbraceright"} to keep track of the sheets in this Hurwitz space.

Next, for each tuple in this list, we apply the monodromy formulae in \Cref{formulainfty,formulaone,formulazero} and check its simultaneous conjugacy class using the function \texttt{sheet\_connections\_with\_markings(degree, HurwSpace)}.
This data corresponds to sheets in the Hurwitz space coming together over $0, 1,$ and $\infty$ in $\Mbar_{0,4}$ (we call such sheets ``connected by an edge'').
We keep track of this information in a dictionary called \texttt{combined\_edges}.
Using this dictionary, we can calculate the connected components of the Hurwitz space via \texttt{find\_connected\_components} which searches for sheets connected to each other by an edge in \texttt{combined\_edges}.
We make a dictionary \texttt{component\_groups} of the components and the sheets in each component.
For each component in \texttt{component\_groups}, the function \texttt{ram\_of\_target\_map} calculates the degree and ramification of the target map  over $0, 1, \infty $ in $\Mbar_{0,4}$.
Finally, we output the genus of the Hurwitz space using the Riemann-Hurwitz formula.
The permutations associated to the ramification of the target map are also provided.

\section{Calculations of $\Hbar_{\underline{d},\underline{g}}(\mu_1, \mu_2, \mu_3, \mu_4)$ for $d \leq 5$ }\label{calculationdata}

Here, we record the genus and degree of the target map on the components of $\Hbar_{\underline{d},\underline{g}}(\mu_1, \mu_2, \mu_3, \mu_4)$ for $d \leq 4$, calculated via the program described in \Cref{python}.
For $d \leq 3$, all the base cases have genus 0, proving \Cref{main}.
As $\Hbar^s_{d,g}$ is a finite group quotient of $\Hbar_{\underline{d}, \underline{g}}(\underline{\mu})$ when the source curve is connected and all $\mu_i = (2,1,\dots, 1)$, this proves that $H^1(\Hbar^s_{3,g}) = 0$ and recovers the fact that $H^1(\Hbar^s_{2,g}) = 0$.

In degree 4, we have three base cases with positive genus, which we list in \Cref{degree4}.
This implies that $H^1(\Hbar_{\underline{d}, \underline{g}} (\underline{\mu}))$ is nonvanishing when $d \geq 4$.
In \Cref{degree5}, we include an example of a postive genus Hurwitz space in degree 5.

In the notation for the Hurwitz spaces, exponential notation indicates repeated ramification profiles. For example, the ramification profiles $((3), (2,1), (2,1), (1,1,1))$ is denoted $((3), (2,1)^2, (1,1,1))$.

\subsection{Degree-2 admissible covers}

\begin{center}
\begin{tabular}{|m{1em}|m{15em} |m{25em} |}
 \hline
 & Space & \# components of genus 0 and degree of the target map \\ [0.5ex]
 \hline\hline
 1 & $\Hbar_{2,}((2)^4)$ &1 w/ deg 1 \\
 \hline
 2 &$\Hbar_{2,}((2)^2, (1,1)^2)$ & 1 w/ deg 2\\
 \hline
 3 & $\Hbar_{2,}((1,1)^4)$ &8 w/ deg 1 \\
 \hline
\end{tabular}
\end{center}

\subsection{Degree-3 admissible covers}

\begin{center}
\begin{tabular}{|m{1em}|m{15em} |m{25em} |}
 \hline
 & Space & \# components of genus 0 and degree of the target map \\ [0.5ex]
 \hline\hline
 1 & $\Hbar_{3,2}((3)^4)$ & 3 w/ deg 1 \\
 \hline
 2 &$\Hbar_{3,1}((3)^3, (1,1,1))$ & 2 w/ deg 1 \\
 \hline
 3 & $\Hbar_{3,1}((3)^2, (2,1)^2)$ & 1  w/ deg 2\\
 \hline
 4 & $\Hbar_{3,0}((3)^2, (1,1,1)^2)$ & 4 w/ deg 3 \\
 \hline
 5 & $\Hbar_{3,0}((3), (2,1)^2, (1,1,1))$ & 1 w/ deg 6\\
 \hline
 6 & $\Hbar_{3,0}((2,1)^4)$ & 1 w/ deg 4 \\
 \hline
 7 & $\Hbar_{(2,1),(0,0)}((2,1)^4)$ & 1 w/ deg 1 \\
 \hline
 8 & $\Hbar_{(2,1),(0,0)}((2,1)^2, (1,1,1)^2)$ & 9 w/ deg 2 \\
 \hline
 9 & $\Hbar_{(1,1,1),(0,0,0)}((1,1,1)^4)$ & 216 w/ deg 1 \\
 \hline
\end{tabular}
\end{center}

\subsection{Degree-4 admissible covers.}\label{degree4}

\begin{center}
\begin{longtable}{|m{1em}|m{17em}|m{23em} |}
 \hline
 & Space & \# components of genus 0 with degree of the target map \\ [0.5ex]
 \hline\hline
 1 & $\Hbar_{4,3}((4)^4)$ & 4 w/ deg 1, 1 w/ deg 4 \\
 \hline
 2 & $\Hbar_{4,2}((4)^2, (2,2)^2)$ & 3 w/ deg 2 \\
 \hline
 3 & $\Hbar_{4,1}((2,2)^4)$ & 3 w/ deg 2 \\
 \hline
 4 & $\Hbar_{(2,2),(0,0)}((2,2)^4)$ & 6 w/ deg 1 \\
 \hline
 5 & $\Hbar_{4,2}((4)^2, (2,2), (3,1))$ & 1 w/ deg 6\\
 \hline
 6 & $\Hbar_{4,1}((4), (2,2)^2, (2,1,1))$ & 2 w/ deg 4\\
 \hline
 7 & $\Hbar_{4,1}((4), (2,2), (3,1),(2,1,1))$ & 1 w/ deg 12  \\
 \hline
 8 &$\Hbar_{4,2}((4)^3, (2,1,1))$ & 2 w/ deg 4 \\
 \hline
 9 & $\Hbar_{4,2}((4)^2, (3,1)^2)$ & 1 w/ deg 2, 1 w/ deg 6\\
 \hline
 10 & $\Hbar_{4,1}((4)^2, (3,1), (1,1,1,1))$ & \textcolor{blue}{1 component of genus 3 mapping with degree 24} \\
 \hline
 \endfirsthead
 \hline
 11 & $\Hbar_{4,1}((4)^2, (2,1,1)^2)$ & 2 w/ deg 8, 2 w/ deg 2 \\
 \hline
 12 & $\Hbar_{4,0}((4)^2, (1,1,1,1)^2)$ & 36 w/ deg 4 \\
 \hline
 13 & $\Hbar_{4,1}((4)^2, (2,2), (1,1,1,1))$ & 6 w/ deg 2 \\
 \hline
 14 & $\Hbar_{4,1}((4), (3,1)^2, (2,1,1))$ & 2 w/ deg 8 \\
 \hline
 15 & $\Hbar_{4,0}((4), (3,1), (2,1,1), (1,1,1,1))$ & 2 w/ deg 24\\
 \hline
 16 & $\Hbar_{4,0}((4), (2,1,1)^3)$ & 2 w/ deg 16 \\
 \hline
 17 & $\Hbar_{4,0}((4), (2,2), (2,1,1), (1,1,1,1))$ & 6 w/ deg 8 \\
 \hline
 18 & $\Hbar_{4,1}((3,1)^4)$ & 3 w/ deg 2, 3 w/ deg 3 \\
 \hline
 19 & $\Hbar_{4,0}((2,2)^3, (1,1,1,1))$ & 12 w/ deg 4 \\
 \hline
 20 & $\Hbar_{4,1}((2,2)^2, (3,1)^2)$ & 2 w/ deg 6 \\
 \hline
 21 & $\Hbar_{4,1}((2,2), (3,1)^3)$ & 2 w/ deg 4\\
 \hline
 22 & $\Hbar_{(2,2),(1,0)}((2,2)^2, (2,1,1)^2)$ & 4 w/ degree 2 \\
 \hline
 23 & $\Hbar_{4,0}((2,2)^2, (2,1,1)^2)$ &  2 w/ deg 8 \\
 \hline
 24 & $\Hbar_{(2,2),(0,0)}((2,2)^2, (1,1,1,1)^2)$ & 144 w/ deg 2 \\
  \hline
 25 & $\Hbar_{4,0}((2,2), (3,1), (2,1,1)^2)$ & \textcolor{blue}{1 component of genus 1, mapping with degree 24}  \\
 \hline
 26 & $\Hbar_{4,0}((2,2), (3,1)^2, (1,1,1,1))$ & 4 w/ deg 12 \\
 \hline
 27 & $\Hbar_{(2,2),(0,0)}((2,2), (2,1,1)^2, (1,1,1,1))$ &12 w/ deg 4\\
 \hline
 28 & $\Hbar_{(3,1),(2,0)}((3,1)^4)$ & 3 w/ deg 1 \\
 \hline
 29 & $\Hbar_{4,0}((3,1)^3, (1,1,1,1))$ & \textcolor{blue}{2 components of genus 1, each mapping with degree 12} \\
 \hline
 30 & $\Hbar_{(3,1),(1,0)}((3,1)^3, (1,1,1,1))$ & 8 w/ deg 1 \\
 \hline
 31 & $\Hbar_{4,0}((3,1)^2, (2,1,1)^2)$ & 2 w/ deg 12 \\
 \hline
 32 & $\Hbar_{(3,1),(1,0)}((3,1)^2, (2,1,1)^2)$ & 4 w/ deg 2 \\
 \hline
 33 & $\Hbar_{(3,1),(0,0)}((3,1)^2, (1,1,1,1)^2)$ & 64 w/ deg 3 \\
 \hline
 34 & $\Hbar_{(3,1),(0,0)}((3,1), (2,1,1)^2,(1,1,1,1))$ & 16 w/ deg 6\\
 \hline
 35 & $\Hbar_{(3,1),(0,0)}((2,1,1)^4)$ & 16 w/ deg 4\\
 \hline
 36 & $\Hbar_{(2,2),(0,0)}((2,1,1)^4)$ & 6 w/ deg 2\\
 \hline
 37 & $\Hbar_{(2,1,1),(1,0,0)}((2,1,1)^4)$ & 8 w/ deg 1\\
 \hline
 38 & $\Hbar_{(2,1,1),(0,0,0)}((2,1,1)^2, (1,1,1,1)^2)$ & 288 w/ deg 2\\
 \hline
 39 & $\Hbar_{(1,1,1,1),(0,0,0,0)}((1,1,1,1)^4)$ & 13824 w/ deg 1\\
 \hline
 \endlastfoot
\end{longtable}
\end{center}
\vspace{-10mm}
\subsection{Degree-5 admissible covers.}\label{degree5}

\begin{example}
  The Hurwitz space $\Hbar_{5,3}(((5)^2, (4,1)^2))$ has 3 components of genus 0, 1 component of genus 1, and 1 component of genus 3.
\end{example}

\bibliographystyle{alpha}
\bibliography{bibliography.bib}

\end{document}